\documentclass[12pt,a4paper]{article}
\usepackage{amssymb}
\usepackage{amsmath}
\usepackage{amsthm}
\usepackage{float}
\usepackage[T1]{fontenc}
\usepackage[utf8]{inputenc}
\usepackage[british]{babel}
\usepackage{geometry}
\usepackage{caption}
\usepackage{graphicx}
\usepackage{color}
\usepackage{nomencl}

\usepackage[all]{xy}

\makeglossary
\makenomenclature

\newcommand{\Hyp}{\mathbb{H}}
\newcommand{\C}{\mathcal{C}}

\newcommand{\R}{\mathbb{R}}
\newcommand{\D}{\mathbb{D}}
\newcommand{\SU}{\mathbb{S}^1}
\newcommand{\N}{\mathbb{N}}
\newcommand{\Z}{\mathbb{Z}}
\newcommand{\F}{\mathcal{F}}

\newcommand{\G}{\Gamma}
\newcommand{\g}{\gamma}

\newcommand{\tv}{\rightarrow}

\newcommand{\Lam}{\mathcal{L}}

\newtheorem{theorem}{Theorem}[section]
\newtheorem{lemme}[theorem]{Lemma}

\newtheorem{prop}[theorem]{Proposition}

\newtheorem{corollaire}[theorem]{Corollary}

\newtheorem{defi}[theorem]{Definition}

\newtheorem*{TheoremNoCount1*}{Theorem A} 
\newtheorem*{TheoremNoCount2*}{Theorem B}

\DeclareMathOperator{\Isom}{Isom}
\DeclareMathOperator{\Card}{Card}

\DeclareMathOperator{\supp}{supp \,}

\DeclareMathOperator{\AdS}{AdS}

\DeclareMathOperator{\Curr}{Curr}

\title{Critical exponent for geodesic currents}
\author{Olivier Glorieux}
\begin{document}
\maketitle
\begin{abstract}
For any geodesic current we associated a quasi-metric space. For a  subclass of geodesic currents, called filling, it defines a metric and we study  the critical exponent associated to this space. We show that is is equal to the exponential growth rate  of the intersection function for closed curves. 
\end{abstract}

\section{Introduction}
\subsection{Geodesic currents}



Let $S$ be a surface of genus greater than 2 and $\G$ its fundamental group.  Let $\D$ be the unit disk and $\Hyp^2$, be $\D$ endowed with the hyperbolic metric. In order to define geodesic current, we fix a hyperbolic metric on $S$ and we see $S$ as the quotient $\Hyp^2/\G_0$ where  $\G_0$ is a fixed discrete subgroup of $\Isom_0(\Hyp^2) $, isomorphic to $\G$.\\
The space of unoriented geodesics of $\Hyp^2$ is identified with $\SU\times  \SU/\Z_2$. 
\begin{defi}
The set of geodesic currents, $\Curr(S)$, is by definition the set of $\G_0$ invariant Borel measures  on the space of geodesics of $\Hyp^2$.  It is endowed with the topology of weak-convergence. 
\end{defi}

Let $\C$ be the set of homotopy class of closed curves on $S$. It is in 1-to-1 correspondence with the set of closed geodesics on $S$ endowed with its hyperbolic metric. The set $\C$ is also in 1-to-1 correspondence with the set of conjugacy classes of $\G_0$.  For any $c\in \C$, we say that $\g\in \G_0$ is a representative of  $c$ if they are in the same conjugacy class. In particular the axis of $\g$ in $\Hyp^2$ projects to the isotopy class of $c$ on $S$. To a closed curve $c\in \C$, we can associate a corresponding current, this is the Dirac measure on the collection of its lifts in $\Hyp^2$. F. Bonahon showed that the set of (weighted) closed curves is dense in $\Curr(S)$.

There is a map $i$ on  $\Curr(S)\times \Curr(S)$, called intersection,  which has been defined by F. Bonahon in \cite{bonahon1986bouts}. This map is  bilinear, continuous and extends the usual geometric intersection to the set of all currents: for a pair of closed curves, $c_1,c_2$, the intersection $i(c_1,c_2)$ between the two corresponding currents is equal to the geometric intersection. Let $G^t$ be the  set of all pair of transverse geodesics. Any two geodesic currents,  $\eta_1,\eta_2$ define a product measure on $G^t$, the intersection of $\eta_1$ and $\eta_2$ is by definition the total mass for this product measure of a fundamental domain for the action of $\G_0$ on $G^t$.

For a closed curve $c\in \C$,  it is  easy to compute its intersection with any other current, $\eta\in \Curr(S)$ as it is explained in \cite{Otal}. Let $A_\g$ be the axis of $\g\in\G_0$ corresponding to $c$ and let $x\in A_\g$ be any point on this axis. Let $T[x,\g x)$ be the set of geodesics transverse to the  $[x,\g x)$, a fundamental domain for the action of $\g$ on $A_\g$. Then $i(\eta,c)=\eta \left(T[x,\g x)\right)$.

Finally we mention that the set of currents contains some interesting  spaces: 

-We can embedded the set of negatively curved metrics on $S$ in $\Curr(S)$.  Indeed  $\Curr(S)$ can be identified to the set of measures invariant by the geodesic flow on $S$. Then, for a metric $m$, we associated the Liouville current which is the local product of the Riemannian metric and the length on the fiber in $T^1(S)$. This map from negatively curved metrics to geodesic currents  is injective \cite{Otal}.  Moreover if we call $L_m$ the Liouville current associated to $m$, it satisfies for all $c\in \C$,  $i(L_m,c)=\ell_m(c),$ where $\ell_m(c)$ is the length of the unique geodesic representative of $c$ on $S$ endowed with the metric $m$.  This shows that the intersection plays the role of a length. That what we are going to use to define a metric space associated to a current.

-As a consequence of the work of G. Mess, \cite{mess2007lorentz} we can also embedded the set of $\AdS$ quasi-Fuchsian (also called globally hyperbolic spacially compact manifolds) in $\Curr(S)$, as the sum of two Liouville currents. The injectivity is a consequence of  a Theorem of Dal'bo and Kim, \cite{dal2000criterion}. For any $\AdS$ quasi-Fuchsian manifold, $M$, it corresponds $(L_1,L_2)$, two points in the Teichmüller space of $S$ by the Mess parametrization. Then, the associated geodesic current $L_{M} := \frac{1}{2} (L_1+L_2)$ satisfies $i(L_M,c) = \ell_M(c)$ where $\ell_M(c)$ is the length (its Riemannian length) of $c$ in $M$, see \cite{glorieux2015behaviour}.

-Finally, we can embedded the set of quasi-Fuchsian hyperbolic manifolds. Indeed, U. Hamenstadt \cite{hamenstadt2002ergodic} showed this by bringing back the Patterson-Sullivan measure of the limit set of the quasi-Fucshian manifold into $\SU$. However, with this construction there is no reason for the associated current, say $\eta_{QF}$ to satisfies, $i(\eta_{QF},c) =\ell_{QF}(c)$ where $\ell_{QF}(c)$ is the length of the unique geodesic corresponding to $c$ in the quasi-Fuchsian manifold. This is an interesting question, whether or not, there exists such a geodesic current.

\subsection{Critical exponent}
If $S$ is endowed with a metric of negative curvature, then one can define two closely related invariants  which have been deeply studied. \\
The critical exponent, $\delta$, defined through the exponential growth rate of the point in an orbit of $\G$: 
$$\delta :=\limsup_{R\tv +\infty} \frac{1}{R}\log \Card \{\g\in \G \, |\, d(\g x,x)\leq R\},$$
where $d$ is the distance on the universal cover of $S$  coming from the chosen metric and $x$ is any point in it. This does not depend on $x$ thanks to triangle inequality. \\
The geodesic critical exponent, $\delta^g$, defined through the exponential growth rate of closed geodesics:
$$\delta^g :=\limsup_{R\tv +\infty} \frac{1}{R}\log \Card \{c\in \C \, |\, \ell(c) \leq R\},$$
where $\C$ is the set of free homotopy classes of closed curves on $S$, and $\ell(c)$ is the length of the unique geodesic on $S$ in the homotopy class. 

For negatively curved metrics, these two invariants are  well known to be equal, see for example \cite{Knieper}. 
Let $L_m$ be the Liouville current associated to the metric  $m$ on $S$. Then from the above, since $i(L_m,c) =\ell_m(c)$, the critical exponent $\delta^g
 =\limsup_{R\tv +\infty} \frac{1}{R}\log \Card \{c\in \C \, |\, i(L_m,c) \leq R\}$
is equal to the growth  rate of the number of closed geodesics on $(S,m)$.
We define critical exponent for  any geodesic current. 
\begin{defi}
For any current $\eta\in \Curr(S)$, we define its \emph{geodesic critical exponent } $\delta^g_\eta$ by 
$$\delta_\eta^g :=\limsup_{R\tv +\infty} \frac{1}{R}\Card \{c\in \C \, |\, i(\eta,c) \leq R\}.$$
\end{defi}


%
%

In Section \ref{sec - Distances for geodesic currents} we define the quasi-distance associated to any geodesic current. We show that the metric is proper for filling currents and therefore the critical exponent is well defined for these currents. 
In Section \ref{sec - critical exponent} we  show that critical exponent is a continuous map on filling currents and coincide with the growth rate of the intersection map for closed geodesics.

\section{Distances for geodesic currents}\label{sec - Distances for geodesic currents}
Our first aim is to define, for every  $\eta\in \Curr(S)$, a distance $d_\eta$ on $\D$. Let us introduce some notations.
\begin{defi}
Let $x,y$ be two distinct points in $\Hyp^2$. 
\begin{itemize}
\item  The set of geodesics transverse to the hyperbolic segment from $x$ to $y$, where $x$ and $y$ \emph{are} included will be denoted by $T[x,y]$.
\item The geodesic passing through $x$ and $y$ will be denoted by $g_{(xy)}$. It does \emph{not} belong to $T[x,y].$ 
\item The set of geodesics passing through $x$ will be denoted by $G(x)$.
\item The set of geodesics transverse to the open segment $(x,y)$ not containing $x$ and $y$ will be denoted by $T(x,y)$. 
\end{itemize}
\end{defi}
With these notations we have $T[x,y]=T(x,y)\cup G(x)\cup G(y)\setminus g_{(x,y)}$.

\begin{defi}
For every current $\eta$ we define $d_\eta$ for all distinct $(x,y)\in \D\times \D$ by 
$$d_\eta (x,y) := \eta(T[x,y]).$$
We set $d_\eta(x,x) =0$ for every $x\in \D$. 
\end{defi}
\paragraph{Remarks}
\begin{itemize}
\item \emph{A priori} $d_\eta$ needs not to be a distance for every $\eta \in \Curr(S)$,  $d_\eta$ does not necessarily separate points. However, it is clear that $d_\eta$ is positive and we will moreover show that $d_\eta$ satisfies the triangle inequality. 
\item Since $\eta$ is $\G_0-$invariant, so is $d_\eta$. 
\item If $d_\eta$ is a distance, then it is Gromov hyperbolic. Indeed, by Svarc-Milnor lemma, using the fact that $\D/\G_0$ is compact, $(\D,d_\eta)$ is quasi-isometric to $\Hyp^2$. 
\item By definition of intersection, if $x$ is on the axis of an element $\g\in\G_0$, then $d_\eta(x,\g x)=i(\eta,c)$ where $c$ is the closed geodesic of $S$ corresponding to $\g$.
\item $d_\eta$ might not be continuous for the usual topology of $\D$. Let $x\in \D$  be a  point, and $G(x)$ be all the geodesics passing to $x$. Suppose that $\eta(G(x))>0$, for example if $\eta$ contains an atom formed by a closed geodesic, and $x$ belongs to  that geodesic. Then let $x_n$ be a sequence of points whose limit is $x$. Then $\lim d_\eta(x_n,x) \geq \eta(G(x))>0 = d_\eta(x,x)$. 
\end{itemize}

\begin{prop}
For all $\eta \in \Curr(S)$, $d_\eta$ satisfies the triangle inequality. 
\end{prop}

\begin{proof}
Let $x,y,z\in \D.$  If $g\in T(x,y)$ then since the triangle $x,y,z$ describes a Jordan curve, one of the intersection between $g$ and the triangle is on $[x,y]$ the other must be on the other segments (two distinct geodesics can intersect only once). Hence $g\in T[x,z]\cup T[z,y]$.\\
Now if $g\in G(x)$, then $g\in T[x,z]$ unless $g=g_{(x,z)}$. In the latter case, $g\in T[z,y]$. \\
 In other words, $T[x,y]\subset  T[x,z]\cup T[z,y]$.
Hence $\eta (T[x,y])\leq  \eta( T[x,z])+\eta (T[z,y])$, this is by definition equivalent to the triangle inequality. 
\end{proof}

\begin{lemme}\label{lem - i < d}
Let $\eta \in \Curr(S)$. Let $c\in \C$ and $\g$ be a representative of $c$. Then for all $x\in \Hyp^2$ we have
$$i(\eta,c) \leq d_\eta(\g x ,x),$$ 
equality occurs for $x\in A_\g$ (but not necessarily only on $A_\g$.)
\end{lemme}
\begin{proof}
Suppose first that the current $\eta$ is a closed curve: $\eta\in \C$, then $i(\eta,c)$ is the geometric intersection of $\eta$ and $c$. By definition, $d_\eta(x,\g x)$ is the number of intersections of the geodesic arc $(x,\g x)/\G$ and $\eta$ on the quotient surface  $\Hyp^2/\G$. Then it is minimal when $c$ is the geodesic representative in its homotopy class. The quotient of the geodesic between $x$ and $\g x$ is a closed curve isotopic to $c$, its geometric intersection is always larger than the intersection of $\eta$ and the  closed geodesic in its isotopy class. By definition the equality occurs on the axis. 

Let $\eta$  be any current and let $\eta_k$ a sequence of closed curves, seen as currents, converging to $\eta$. We have 
$$i(\eta_k,c) \leq  d_\eta(\g x, x)=\eta_k(T[\g x, x])$$
Taking the $\limsup$ we get 
$$i(\eta,c)=\limsup i(\eta_k,c)\leq \limsup \eta_k(T[\g x, x])\leq \eta(T[\g x, x])=d_\eta(\g x, x).$$
\end{proof}

\subsection{Filling currents}
In general, the set $\{c\in \C \, |\, i(\eta,c) \leq R\}$ might be infinite that why we restrict ourselves to \emph{filling} currents, for which this invariant is finite. We use the following definition, inspired by the example given after \cite[Proposition 4]{bonahon1988geometry} :
\begin{defi}\label{def filling}
A current $\eta \in \Curr(S)$ is said to be \emph{filling}, if every geodesic of $\Hyp^2$ transversely meets another geodesic which is in the support of $\eta$. 
 The set of filling currents will be denoted by $\F$. 
\end{defi} 
\paragraph{Remarks} The property of being filling is not equivalent to the property of having strictly positive intersection with all closed geodesics. Indeed a lamination $\Lam$ is never filling since its auto-intersection is 0, but if it is maximal $i(c,\Lam)>0 $ for all $c\in \C$. However, we will show in Proposition \ref{pr F filling iff i(F,eta)>0}, that a current is filling if it has strictly positive intersection with \emph{all} currents.

Recall this compactness theorem of F. Bonahon \cite[Proposition 4]{bonahon1988geometry} 
\begin{theorem}
For every $\eta \in \F$ the set 
$$\{\mu\in \Curr(S) \, |\, i(\eta,\mu) \leq 1\},$$
is compact.
\end{theorem} 
It has a nice corollary: 
\begin{corollaire}\label{cor - inequality between length}
For every filling current $\eta \in \F$, there exists $K>0$ such that for all $c\in \C$ we have : 
$$\frac{1}{K} i(\eta,c)\leq \ell(c) \leq K i(\eta,c).$$
\end{corollaire}
\begin{proof}
Recall that $\ell(c) = i(L,c)$ where $L$ is the Liouville current associated to the hyperbolic metric on $S$. The function $c\tv i(L,c)$ is a continuous function on $\Curr(S)$, hence it is bounded on $\{c\in \Curr(S) \, |\, i(\eta,c) = 1\}$.
By projectivization, it implies that the function $\frac{i(L,c)}{i(\eta,c)}$ is bounded on $\Curr(S)$, so in particular on $\C$. 
\end{proof}
 For any filling current, the set $\{c\in \C \, |\, i(\eta,c) \leq R\}$ is hence finite for all $R>0$ by the previous compactness remark. Therefore for all $\eta\in \F$, $\delta_\eta^g<+\infty$ 
 

We are going to prove that if $\F$ is filling then $d_\eta$ is proper.



\begin{prop}\label{pr - properness}
Let $x\in \Hyp^2$. Let $\eta\in \Curr(S)$ and for $R\geq 0$, we let  $B_\eta(x,0):=\{y\in \Hyp^2 \, | \, d_\eta(x,y) =0\},$ be the ball radius R, then the following are equivalent: 
\begin{enumerate}
\item $\eta\in \F.$
\item $\forall x\in \D, \,  B_\eta(x,0)\text{ is compact}.$
\item $\forall x\in \D, \, \text{and }\forall R\geq 0, \,  B_\eta(x,R)\text{ is compact}$ ie. $d_\eta$ is proper. 
\end{enumerate}

\end{prop}
\begin{proof}
$3\Rightarrow 2$ is obvious. \\
$2 \Rightarrow 1$. If $\eta$ is not filling there exists $g$ a geodesic of $\Hyp^2$ which does not intersects the support of $\eta$. Then for all $x,y\in g$ we have $d_\eta(x,y)=0$ by definition of $d_\eta$, hence for all $x\in g$,  $B(x,0)\supset g$ hence it is not compact. \\
$1 \Rightarrow 3$
Suppose now that $\eta$ is filling. For every $v\in T^1S$  (endowed with the fixed hyperbolic metric) we define $r(v) :=\inf\{t\, |\, \pi \phi_t(v) \pitchfork \supp(\eta) \neq \emptyset\},$ to be the first time where the geodesic issued form $v$ intersects the support of the geodesic current transversally. Here $\pi $ is the canonical projection $T^1S \tv S$ and $\phi_t$ is the geodesic flow on $T^1S$. Since $\eta$ is filling $r$ is  finite. \\
As it is easy to see on the universal cover, $r$ is upper semi-continuous \footnote{it is not continuous because of  the transversality condition : for example if the geodesic issued by $v$ is on the support.}, hence $r$ admits an upper bound C on the compact $T^1S$. By $\G$ invariance, $r$ can be lift to a  function on $T^1\Hyp^2$, bounded by $C$. As a consequence, let $[x,y]$ be a segment of $\Hyp^2$ length larger than $C+1$, then the set of transverse geodesics $T[x,y]$ intersects transversally the support of $\eta$, hence $\eta(T[x,y])>0$. By compactness, there is $\epsilon>0$ independent of $x,y$ such that $\eta(T[x,y]) \geq \epsilon.$ In particular, this implies that $B_\eta(x,\epsilon )$ is compact. By induction, any segment of $\Hyp^2$ length larger than $n(C+1)$ is of $d_\eta$ length larger than $\epsilon n$. This implies the compactness of $B_\eta(x,n\epsilon)$ for all $n\in \N$, and concludes the proof. 
\end{proof}

 \begin{defi}
For every geodesic current $\eta \in \Curr(S)$, we define the metric critical exponent by 
$$\delta_\eta :=\limsup_{R\tv +\infty} \frac{1}{R}\Card \{\g \in \G_0 \, |\, d_\eta(\g  o , o ) \leq R\}.$$
\end{defi}
Since $\G_0$ acts properly discontinuously, we deduce from Proposition \ref{pr - properness}: 
\begin{corollaire}
$\delta_\eta$ is finite if and only if $\eta \in \F$. 
\end{corollaire}

Remark that most currents are filling: 
\begin{prop}\label{th - F is open}
$\F$ is a dense subset of  $\Curr(S)$.
\end{prop}

\begin{proof}
 Let $\eta\in \Curr(S)$ and $\mu\in \F$. Consider the current $\eta_\epsilon=\eta + \epsilon \mu$. 
Clearly $\eta_\epsilon \in \F$ and converges to $\eta$ as $\epsilon\tv 0$. 
\end{proof}

We will use the following characterisation of filling currents, to show that $\F$ is open. We saw that having strictly positive intersection with all closed curves is not a sufficient condition to be filling. We are going to show however that having strictly positive intersection with \emph{all currents} is equivalent of being filling. 
\begin{prop}\label{pr F filling iff i(F,eta)>0}
$F\in \F$ if and only if for all $\eta \in \Curr(S)$, $i(F,\eta)>0$.
\end{prop}
\begin{proof}
Suppose there is $\eta\in \Curr(S)$, such that $i(F,\eta)=0$. Then for all $n\in \N$, $i(F,n \eta)=0$, hence the set $\{c\in \Curr(S)\, |\, i(F,c)\leq 1\}$, is not compact. According to Bonahon's theorem, \cite[Proposition 4]{bonahon1988geometry}, $F$ is not filling. 

Conversely, suppose that $F$ is not filling, we have to construct a current whose intersection with $F$ is 0. We know there exists a geodesic $g$ which does not intersect $\supp(F)$ transversally. Then either $g\cap \supp(F)=\emptyset$ or $g\subset \supp(F)$. \\
In the first case, since closed geodesics are dense in the set of geodesics and $\supp(F)$ is closed, we can find a closed geodesic $c$ such that $c\cap \supp(F)=\emptyset$. This means that the geodesic current associated to $c$ satisfies $i(F,c)=0$.\\
If $g\subset \supp(F)$, then by density there is no geodesic on the support of $F$ transverse to  the closure of $g$, $\overline{g}$. By Bogolyubov-Krilov theorem, there exists an invariant measure for the geodesic flow whose support is included in $\overline{g}$. Let $\alpha$ the associated geodesic current. There is no pair $(g_1,g_2)$ of transverse geodesics, for which $g_1\in \supp(F)$ and $g_2\in \supp(\alpha)$. By definition of the intersection, it implies that $i(F,\alpha)=0$.
\end{proof}

\begin{corollaire}
$\F$ is open in $\Curr(S)$. 
\end{corollaire}
\begin{proof}
Let $\eta_k$ a converging sequence of non-filling geodesic current. We call the limit $\eta_\infty$. There exists $\lambda_k$ such that $i(\eta_k,\lambda_k)=0$. Up to normalisation and subsequence, we can suppose that $\lambda_k$ converges to $\lambda_\infty$. By continuity we have, $i(\eta_\infty,\lambda_\infty)=0$, which implies, according to the previous lemma that $\eta_\infty$ is not filling. 
\end{proof}

 \section{Critical exponents}\label{sec - critical exponent}
The aim of this section is to show the following Theorem:
\begin{theorem}\label{th delta = deltag}
For all $\eta\in \Curr(S)$ we have: $\delta_\eta=\delta_\eta^g. $  It is finite if and only if $\eta$ is filling (cf. Definition \ref{def filling}). And, restricted to filling currents, it defines a continuous function.
\end{theorem}

%


\begin{proof}[\textbf{Proof of continuity} ]
 Let $\eta_\infty$ be a filling current and let $F\, : \, \Curr(S) \tv \R$ be the  function defined by 
$$F(\eta )  :=\sup_{c\in \Curr_1(S)} \frac{i(\eta,c)}{i(\eta_\infty,c)}$$
where $\Curr_1(S)$ is the compact set of currents whose intersection with $\eta_\infty$ is equal to $1$. By the latter compacity, $F$ is continuous and $F(\eta_\infty)=1$. Then, the continuity  of $\eta\tv \delta_\eta^g$ follows from the inequality: 
$$\delta^g_{\eta_\infty} \min F \leq \delta^g_\eta \leq \delta^g_{\eta_\infty} \max F.$$
\end{proof}

We fix once for all $\eta\in \F$. We follow the proof of \cite{Knieper} which is  very flexible. We will prove separately the two inequalities $\delta_\eta^g \leq \delta_\eta$ and then $\delta_\eta \leq \delta^g_\eta$.

We first prove that $\delta_\eta^g \leq \delta_\eta$

\begin{proof}[\textbf{Proof of $\delta_\eta^g \leq \delta_\eta$}]
Let $\G_\eta(R) := \{ \g \in \G \, |\, d_\eta(\g o , o ) \leq R\},$ and   $\C_\eta(R) := \{ c \in \C \, |\, i(\eta,c)\leq R\}.$ Fix also $N$ a compact fundamental domain of $\G$ acting on $\Hyp^2$ containing $o$. Since $N$ is compact the $\eta$ measure of all geodesics intersecting $N$ is bounded, says by $K>0$. 
Let $j : \C \tv \G$ the inclusion map  sending $c$ to $\g\in\G$ whose axis $A_\g$ intersects $N$. Let $p \in A_\g\cap N$.
\begin{eqnarray}
d_\eta(\g o , o )&\leq & d_\eta(\g p , p ) + 2d_\eta(o,p).\\
						  &\leq & i(\eta , c) + 2K
\end{eqnarray}
Hence $j$ restrict to a injection from $\C_\eta(R)$ to $\G_\eta(R+2K)$. In particular 
$$\delta_\eta^g \leq \delta_\eta.$$
\end{proof}
The other inequality is more demanding. First we prove that the projection
$$\pi : \G\tv \C,$$
associating to any element $\g\in \G$  its conjugacy class $[\g]$ induces a map from $\G_\eta(R)$ onto $\C_\eta(R)$.

We look at the subset of $\G_\eta(R)$ whose axis is at bounded distance $D>0$ of $o\in \Hyp^2$. 
$$B_\eta(R) := \{\g \in \G \, | \, d(A_\g,o) \leq D \text{ and } i(\eta, [\g]) \leq R\}.$$
We will first prove that the cardinal of $\pi^{-1}([\g])\cap B_\eta(R)$   is polynomial in $R$. Then we will show that  $B_\eta(R)$ has the same exponential rate as $\G_\eta(R)$. 

By Lemma \ref{lem - i < d},  if $D$ is larger than the diameter of $N$,(for the hyperbolic distance)  then the map $\pi : B_\eta(R) \tv \C_\eta(R)$ is onto. 

\begin{lemme}\label{lem- 1}
There exists $K>0$ such that for all $R>0$ and all $[\g]\in \C_\eta(R),$
$$\Card\left( \pi^{-1} ([\g])\cap B_\eta(R) \right)\leq KR.$$
In particular, $\Card \C_\eta(R) \geq \frac{1}{KR}\Card B_\eta(R).$
\end{lemme}

\begin{proof}

Fix $[\g] \in \C_\eta(R)$, $\g\in \G$ a representative of $[\g]$ and  $A_\g$ the axis of $\g$. Let $E:= \{ g \in [\g] \, | \, d(A_g,o)\leq D \}$.

For the sake of completeness we reproduce the proof of G. Knieper showing that there exists $K>0$ such that. 
$$\Card (E) \leq K \ell([\g]).$$
Then, the lemma will follow, using the inequality of Corollary \ref{cor - inequality between length} : $\ell([\g]) \leq K' i(\eta , [\g])$.

Let $s$ be the length of the smallest geodesic on $S$, and consider $r :=s /4$. For all $p\in \Hyp^2$, we will find an upper bound of 
$$F_{p}= \{ g \in [\g] \, | \, d(A_g,p)\leq r \}.$$
By compactness of $N$,  There exists a finite set of balls $\left(B(p_i,r)\right)_{i\in [1,n]}$ such that $N\subset \cup_i B(p_i,r)$. Therefore  $\Card (E) \leq n\max_{i}\Card(F_{p})$.It is sufficient to find a polynomial upper bound for $F_p$ to prove the lemma. 

Let  $g$ and $g'$ be two distinct elements of $F_{p}$.  By definition, there exists $\beta$ and $\beta'$ such that $g= \beta \g \beta^{-1}$ and $g' = \beta' \g \beta'^{-1}$. We have $A_g=\beta(A_\g)$ et   $A_{g'} =\beta'(A_\g)$. Fix a fundamental domain $I$ for the action of $\g$ on $A_\g$. The length of $I$ is equal to $\ell(\g)$. Since $g,g'\in F_{p}$, by definition, there exists $q$ and $q'$ in $I$ and $n,n' \in \Z$ such that 
$$\beta\g^n q \in B(p,r) \quad \text{ and } \quad \beta'\g^{n'} q' \in B(p,r). $$
Also $q,q'\in I$, implies 
\begin{eqnarray}\label{eq - interne lemme 1}
d(q,q')\leq \ell([\g]).
\end{eqnarray}
We are going to find a lower bound for $d(q,q')$, which will therefore induce an upper bound for $F_{p}$. We call $\alpha:=\beta \g^n $ and $\alpha':=\beta'   \g^{n'}$. Since $g\neq g'$, we have $\alpha\neq \alpha'$.  By definition of $s$, we have
\begin{eqnarray*}
s    &\leq & d(\alpha'^{-1} \alpha q ,q) = d(\alpha q ,\alpha'q) \\
      &\leq & d(\alpha q,\alpha'q') + d(\alpha'q,\alpha'q') \\
      &\leq & 2r +d(q,q') = s/2 +d(q,q').
\end{eqnarray*}
In particular the points $q$ and $q'$ satisfies 
 \begin{eqnarray}\label{eq2 - interne lem 1}
d(q,q') > s/2.
\end{eqnarray} 
Equations (\ref{eq - interne lemme 1}) and (\ref{eq2 - interne lem 1}) shows that $\Card F_{p} \leq \frac{2\ell(I)}{s}$. Finally, there exists $K>0$ such that 
\begin{eqnarray}\label{eq 1}
\Card E \leq K \ell([\g]).
\end{eqnarray} 

\end{proof}

We are going to show that the cardinal of $B_\eta(R) $ has the same exponential growth as $\G_\eta(R)$.
We will show this in two times. We choose two open sets $U,V$ of $\SU=\partial \Hyp^2$, such that every geodesics whose endpoints lie in $U$ and $V$ intersect $N$. First we show that if an element $\g\in \G_\eta(R)$ satisfies $\g U \cap V\neq \emptyset $ then we can bring back its axis in $N$ without changing to much its translation length. Then we show that the elements of $\G_\eta(R)$ which does not satisfy this property are very few, this allows to find an upper bound for $B_\eta(R)$. 

The two following lemmas are very classical and can be found in \cite{}: 
\begin{lemme}\label{lemme géométrique 1}
There exists  $U,V\subset \SU$ such that for any $u \in U, v \in V$, we have  $d ((u,v), o) \leq D$. Moreover every subset of $U$ and $V$ have the same property. 
\end{lemme}
We fix once for all an element  $r\in \G$ whose axis $A_r$ intersects $N$. This element will be used to normalize elements of $\G$ in order to make their axis intersect $N$. We denote by $r^\pm$ the repulsive and attractive fixed point of $r$. 
\begin{lemme}\label{lem topo 1}
For all neighbourhood  $U\subset \partial \Hyp^2$ of  $r^+$ and $V\subset \partial \Hyp^2$ of $r^-$ there exists $n\in \N$ such that 
$$r^n (\partial \Hyp^2\setminus V) \subset  U,$$
$$r^{-n} (\partial \Hyp^2\setminus U) \subset  V.$$
\end{lemme}
We show the lower bound for $B_\eta(R)$
\begin{lemme}\label{lem  - g U1 cap V1 =empty set < B(R+c)}
For every neighbourhood  $U$ of $r^-$ and $V$ of $r^+$, there exists $c>0$ such that for all $R>0$ we have
$$\Card\{ \g \in \G_\eta(R) \, | \, \g U \cap V = \emptyset \} \leq \Card B_\eta(R+c).$$
\end{lemme}

 \begin{proof}
Remark that if two open sets $U,V$ satisfy  the conclusion of the Lemma \ref{lem  - g U1 cap V1 =empty set < B(R+c)}, then every other open sets containing  $U$ and $V$ also satisfy the lemma with the same $c$. It is then sufficient to find to prove it for small open sets. 

Let  $U$ and $V$ as in Lemma \ref{lemme géométrique 1}.  Without loss of generality we can take smaller $U$ and $V$, such that $ V\subset \partial \Hyp^2\setminus U $ and $U \subset \partial \Hyp^2 \setminus V$.  Let $n$  be as in Lemma \ref{lem topo 1} and consider the map 
$f(\g) = r^n \g r^n$, which "normalizes" the elements  of $\G$. For all $\g\in \G$ such that $\g U \cap V = \emptyset $, we have  
$$f(\g) U \subset U$$
and
$$f(\g) ^{-1}V \subset V.$$ 
In particular the fixed points of $f(\g)$ are in $U$ and $V$. Therefore $A_{f(\g)}$ intersects $N$.

We now estimate  $i(\eta,[f(\g)])$. Let $q$ be a point on $A_{f(\g)}\cap N$.  We have 
\begin{eqnarray*}
i(\eta,[f(\g)])&=& d_\eta( (r^n \g r^n) q,q) \\
				&\leq & d_\eta (\g o,o)+2d_\eta(o,q) +2 d_\eta(r^n q,q)
\end{eqnarray*}
The lemma follows with $c= +2d_\eta(o,q) +2 d_\eta(r^n q,q)$
 \end{proof}

\begin{lemme}\label{lem -3}
There exists open set $U$ and $V$ such that for all $R>0$ we have 
$$\Card\{\g \in \G_\eta(R) \, | \, \g U \cap V =\emptyset \} \geq \frac{1}{2} \left( \Card \G (R) \right) $$
\end{lemme}		

\begin{proof}
We define the set  
 $$A_\eta(U,V,R) := \{ \g \in \G_\eta(R) \, | \, \g U \cap V = \emptyset \}$$ 
 and
 $$A_\eta(R) := A_\eta(U,V,R) \cup A_\eta(V,U,R).$$
Remark that $\Card A_\eta(U,V,R)  =\Card A_\eta(V,U,R)$, since the inverse function is 1-to-1 between the two sets. 
Therefore, $\Card( A_\eta(R)) \leq 2 \Card A_\eta(U,V,R)$. We are going to show that we have in fact $A_\eta(R)=\G_\eta(R)$.  Let $\g\in \G_\eta(R)\setminus A_\eta(R)$. It must satisfy the two following conditions: 
\begin{eqnarray}
\g U \cap V \neq \emptyset\\
\g V \cap U \neq \emptyset.
\end{eqnarray} 
Without loss of generality, we suppose that $V\cap U =\emptyset$. We are going to show these two conditions cannot be simultaneously satisfied:
 \begin{itemize}
 \item If the attractive point of  $\g$ is in  $U$ (resp. in  $V$), then $\g U \subset U$ (resp. $\g V \subset V$) and the first (resp. the second) condition cannot be fulfilled.
 \item If the repulsive point of $\g$ is in  $U$ and the attractive point is outside $U\cup V$. Then $\g V$ is closer to the attractive point and therefore $\g V \cap  U =\emptyset$. The same argument works if the repulsive point of  $\g$ is in $V$. 
 \item If no fixed point is in $U\cup V$, there are two possibilites. The axis of $\g$ is in a connected component of  $\SU\setminus \left( U\cup V\right)$ or the axis cuts$\SU$ in two connected components, one containing $U$ the other containing $V$. In the first case, since $\g$ preserves the orientation, at least one of the condition cannot be fulfilled. In the second case, as the action on the boundary is continuous, it preserve connected components and neither of them can be fulfilled. 
 \end{itemize}
This shows that $\G_\eta(R)\setminus A_\eta(R) =\emptyset$.  
$$\G_\eta(R)=A_\eta(R),$$
this finishes the proof. 

\end{proof}

\bibliographystyle{alpha}

\end{document}